\documentclass[12pt, reqno]{amsart}
\usepackage{amsmath}
\usepackage{amssymb}
 2

\newtheorem{thm}{Theorem}[section]
\newtheorem{lem}[thm]{Lemma}

\newtheorem{prop}[thm]{Proposition}
\newtheorem{cor}[thm]{Corollary}
\newtheorem{defn}[thm]{Definition}

\newcommand{\thmref}[1]{Theorem~\ref{#1}}
\newcommand{\propref}[1]{Proposition~\ref{#1}}
\newcommand{\lemref}[1]{Lemma~\ref{#1}}

\newcommand{\corref}[1]{Corollary~\ref{#1}}
\theoremstyle{remark}

\newcommand{\Z}{{\mathbb Z}}
\newcommand{\Q}{{\mathbb Q}}

\newcommand{\R}{{\mathbb R}}
\newcommand{\C}{{\mathbb C}}

\newcommand{\K}{{\mathbb K}}

\begin{document}

\title[Number field extension]{A number field extension of \\
a question of Milnor}

\author{T. Chatterjee, S. Gun and P. Rath}

\address[Tapas Chatterjee]
{Indian Institute of Technology Ropar,
 Nangal Road,  Rupnagar 140001, 
 Punjab, 
 India}

\address[Sanoli Gun]
      {The Institute of Mathematical Sciences,
       CIT Campus, Taramani, Chennai 600113,
       India.}
\address[Purusottam Rath]
      {Chennai Mathematical Institute,
       Plot No H1, SIPCOT IT Park,
       Padur PO, Siruseri 603103,
       Tamil Nadu, 
       India.}      

\email[T. Chatterjee]{tapasc@iitrpr.ac.in}

\email[S. Gun]{sanoli@imsc.res.in}
\email[P. Rath]{rath@cmi.ac.in}

\maketitle

{\sl \hfill To Professor Ram Murty on the occasion of his 
sixtieth birthday \qquad }

\begin{abstract}
Milnor \cite{JM} formulated  a conjecture about 
rational linear independence of some special 
Hurwitz zeta values. In \cite{GMR}, this conjecture was
studied and an extension of Milnor's conjecture was suggested.
In this note, we investigate the number field generalisation
of this extended Milnor conjecture. We indicate the motivation
for considering this number field case by noting 
that such a phenomenon is true in an analogous context. 
We also study  some new spaces related to 
normalised Hurwitz zeta values.
\end{abstract}

\section{Introduction}

For a real number $x$ with $ 0 < x \le 1$ and $s\in \C$ with $\Re(s)>1$,
the Hurwitz zeta function is defined by
\begin{equation*}
\zeta(s,x):=\sum_{n=0}^\infty \frac{1}{(n+x)^s}.
\end{equation*}

This (as a function of $s$) can be analytically extended  to the entire
complex plane except at $s=1$ where it has a simple pole
with residue one. Note that $\zeta(s,1)=\zeta(s)$ is
the classical Riemann zeta function.

In 1983,  Milnor (see \cite{JM}, \S 6) made a conjecture about the 
linear independence of certain special Hurwitz zeta values over $\Q$. 
More precisely, he suggested the following:

{\it 
For integers $q, k>1$, the $\Q$-linear space
$V(k,\Q)$ generated by the  real numbers
\begin{equation*}
\zeta(k,a/q), ~1\le a< q ~{\rm with} ~(a,q)=1  
\end{equation*}
has dimension $\varphi(q)$.
}

The relevance of  these Hurwitz zeta values is that they 
form a natural generating set
for the study of special values of  Dirichlet series 
associated to periodic arithmetic
functions. More precisely, one is interested in the 
special values of $L$-series of the form
$$
L(s,f) := \sum_{n=1}^{\infty} \frac{f(n)}{n^s}
$$ 
where $f$ is defined over integers and  $f(n+q) = f(n)$ for all integers
$n$ with a fixed modulus $q$. Typically, $f$ takes algebraic values.
Running over arithmetic progressions mod $q$, 
one immediately deduces that 
$$
L(s,f) = q^{-s} \sum_{a=1}^{q} f(a) \zeta(s,a/q).
$$ 

In \cite{GMR}, the second and the third author with M. Ram Murty 
studied Milnor's  conjecture and derived
a non-trivial lower bound for the dimension
of $V(k,\Q)$, namely that the dimension is at least 
half of the conjectured dimension. They also obtained
a conditional improvement of this lower bound and noted
that any unconditional improvement of this ``half''
threshold will have remarkable consequences in relation to irrationality
of the numbers $\zeta(2d+1)/{\pi}^{2d+1}$.

Furthermore in \cite{GMR}, the authors suggested 
a generalisation of the original conjecture of Milnor. There are 
at least two reasons for considering
such a generalisation. First is that the 
inhomogeneous version of Baker's theorem
for linear forms in logarithms of algebraic 
numbers naturally suggests such
a  generalisation. Secondly, typically one is interested in
irrationality of $\zeta(2d+1)/{\pi}^{2d+1}$ as well as that of
$\zeta(2d+1)$ and this generalisation predicts the irrationality
of both these numbers. Following is this extension suggested by 
the authors (see \cite{GMR}):

Extended Milnor conjecture: 
{\it In addition to the original Milnor's conjecture, 
$V(k,\Q) \cap \Q = \{0\}$.}

In an earlier work \cite{TC}, the first  author 
considered various
ramifications of this conjecture.

In this work, we investigate the number field
extension of the above conjecture. One 
of the reasons for considering such an extension
is that  we are interested in the transcendence of
odd zeta values $\zeta(2d+1)$ as well
as of the normalised values $\zeta(2d+1)/{\pi}^{2d+1}$. 
This extension predicts such an eventuality.
Moreover, there is a related set up 
where the analogous statement can be 
established unconditionally. This is
the content of \thmref{req} in the next section.
See also \cite{GMR2} and  \cite{WK} for a modular interpretation
of the conjectural transcendence of
the normalised values $\zeta(2d+1)/{\pi}^{2d+1}$. 

It will be evident  that considering the  
extended Milnor conjecture to a number field $\K$ comes
with a caveat, namely it depends on the arithmetic of $\K$
(for instance compare \thmref{3} with \corref{1}).

As we shall see in section 3, the expected $\K$-dimension 
is $\varphi(q)$ for number fields $\K$ such that
$\K \cap \Q(\zeta_q) = \Q$. In such cases, the
mathematics is somewhat amenable and one can derive
similar lower bounds for these dimensions as has been
done in the earlier works \cite{GMR} and \cite{GMR1}.

On the other hand, when the ambient number field $\K$
has non-trivial intersection with  the $q$-th cyclotomic field $\Q(\zeta_q)$,
nothing is known. In section 4, we investigate
this difficult case and derive some results. We also
try to highlight  the  crux of the complexity.

Finally in the last section, we consider some new 
spaces generated by normalised
Hurwitz zeta values which appear naturally in the study of
irrationality of odd zeta values. The 
mathematics in this set up is somewhat different. For instance,
the parity of $k$ enters into the question non-trivially which is not evident
in the earlier questions.

\section{The analogous case for the space generated 
by the values of $L(1,\chi)$}

In this section, we consider the question of linear independence of
the special values $L(1,\chi)$ as $\chi$ runs over non-trivial Dirichlet
characters mod $q$. This  serves as a guiding line for the questions addressed in this work.

One of reasons why we have a clearer picture in this context
is the following seminal theorem of Baker (see \cite{AB}, also \cite{MR}).

\begin{thm}
{\it If $\alpha_1,
\cdots \alpha_n$ are non-zero algebraic numbers such that the numbers
$\log \alpha_1,
\cdots,\log \alpha_n$ are linearly independent over rationals,
then the numbers $1, \log \alpha_1,
\cdots,\log \alpha_n$ are linearly independent over $\overline{\Q}$.}
\end{thm}

In an earlier work,  R. Murty and K. Murty \cite{RK}  used Ramachandra units to prove that
the values  $L(1,\chi)$ as $\chi$ runs through non-trivial even Dirichlet
characters mod $q$ are linearly independent over $\overline{\Q}$.
We note that without much effort, the following extension of their
result can be obtained.

\begin{thm}\label{req}
The numbers $L(1,\chi)$ as $\chi$ runs through non-trivial even Dirichlet
characters mod $q$ and $1$ are linearly independent over $\overline{\Q}$.
\end{thm}
\begin{proof}
As noticed in \cite{RK}, each of these special values is a linear
form in logarithms involving real multiplicatively independent
units of Ramachandra. Thus any linear combination
$$
\sum_{\chi \text{ even } \atop \chi \ne 1} \lambda_{\chi} L(1, \chi)
$$
with $\lambda_{\chi}$ algebraic, not all zero, 
is necessarily transcendental
by Baker's theorem.
\end{proof}

We now  highlight as well as summarise the salient features 
in this set up. This will serve as an indicator of what to expect
in the more involved case of special values related
to Milnor's conjecture.

\begin{itemize}

\item{} When $\chi$ is an odd character, it can seen that $L(1,\chi)$
is an algebraic multiple of $\pi$ (see page 38 of \cite{LW} for instance).
 Thus the $L(1,\chi)$
values when $\chi$ runs through  odd characters mod $q$
form a one dimensional vector space over $\overline{\Q}$.
Let us call this space the arithmetic space and denote it
by $V_{ar}$. Since $\pi$ is transcendental, we have
$$
V_{ar} \cap \overline{\Q} = \{ 0 \}.
$$ 

\item{} 
The $\overline{\Q}$ vector space generated
by the  $L(1,\chi)$
values when $\chi$ runs through non trivial 
even characters mod $q$
is of optimal dimension $\varphi(q)/2 - 1$. 
Let us call this space the 
transcendental space and denote it by $V_{tr}$.
If we assume Schanuel's conjecture, all these values are 
algebraically independent. Recall that Schanuel's 
conjecture (see \cite{MR}, page 111) 
is the assertion that  for any collection of 
complex numbers $\alpha_1,  \cdots , \alpha_n$
that are linearly independent over $\Q$, the 
transcendence degree of the field
$$
{\Q}(\alpha_1,  \cdots , \alpha_n,  
e^{\alpha_1}, \cdots , e^{\alpha_n})
$$
over $\Q$ is at least $n$. 
\\

\item{} The transcendental space 
intersects $\overline{\Q}$ trivially, that is, 
$$ 
V_{tr} \cap \overline{\Q} = \{ 0 \}.
$$ 
This follows from \thmref{req}.
\\

\item{} Finally, we can  prove the following 
stronger assertion, namely that
the following sum 
$$
V_{ar} + V_{tr}+\overline{\Q}
$$ 
is direct.

Here is a proof  of this assertion. The  values
of $L(1,\chi)$ for non-trivial even characters $\chi$ are linear forms
in logarithms of real positive algebraic numbers. On the other hand,
 when $\chi$ is an odd character, $L(1,\chi)$ 
is an algebraic multiple of $\log(-1)$.  By Baker's theorem, 
any $\overline{\Q}$-relation involving 
logarithms of positive real algebraic 
numbers (from non-trivial even characters)  
and $\log (-1)$ will result in a
$\Z$-linear relation between these numbers. This will lead
to  a contradiction as $\log(-1) = i \pi$ is purely imaginary. 
This along with Theorem 2.2
proves that the above sum is direct.

\end{itemize}

\section{generalised Milnor conjecture over 
number fields intersecting $\Q(\zeta_q)$ trivially}

Let us first set some notations in relation to the extended Milnor conjecture
over number fields.
Let $\K$ be a number field and $k  > 1, q >2$ be integers. Let 
$\widehat{V}_k(q,\K)$ be the $\K$-linear space generated by the numbers
$$
1, ~\zeta(k,a/q), ~1\le a< q ~{\rm with} ~(a,q)=1.  
$$
We are interested in the dimension of this space. This as 
we shall see will depend
on the chosen number field $\K$. We first isolate 
the following two canonical subspaces 
of $\widehat{V}_k(q,\K)$, namely the $\K$-linear 
space spanned by the following sets of real numbers:
\begin{equation*}
\left\{\zeta(k,a/q)+(-1)^k \zeta(k,1-a/q) ~:~ (a,q)=1, ~ 1\leq a<q/2\right\},
\end{equation*}
which we refer to as the ``arithmetic  space'' and the space spanned by
\begin{equation*}
 \left\{\zeta(k,a/q)+ (-1)^{k+1} \zeta(k,1-a/q)  ~:~  (a,q)=1, ~ 1\leq a<q/2\right\}
\end{equation*}
which we call the ``transcendental space''. Let us denote them by $V_{ar}(\K)$
and $V_{tr}(\K)$ respectively.

We now state the following results which are of relevance in this set up.
First, one has  the following theorem of Okada \cite{TO} (see also \cite{MS}).

\begin{lem}\label{lem1}
Let $k$ and $q$ be positive integers with $k > 0$ and $q > 2$. 
Let $\rm T$ be a set
of $\varphi(q)/2$ representations mod $q$ such that the 
union ${\rm T} \cup (-{\rm T})$
constitutes a complete set of co-prime residue classes mod $q$. 
Let $\K$ be a number field such that ${\K} \cap \Q(\zeta_q) = \Q$.
Then the set of real numbers
$$
\frac{d^{k-1}}{dz^{k-1}} \cot (\pi z)|_{z=a/q}, ~~~~~~~~~~a \in T
$$
is linearly independent over $\K$.
\end{lem}

We shall be frequently using the following identity (see \cite{MS}, for instance):
\begin{equation}\label{strong}
\zeta(k, a/q) + (-1)^k \zeta(k, 1-a/q) = \frac{(-1)^{k-1}}{(k-1)!} 
~ \frac{d^{k-1}}{dz^{k-1}} (\pi \cot \pi z) |_{z = a/q}.
\end{equation}

Finally,  one has the following result established in \cite{GMR1}:
\begin{lem}\label{lem2}
For any $ 1 \leq a < q/2$ with $(a,q) = 1$, the number 
$$
\frac{\zeta(k, a/q) + (-1)^k \zeta(k, 1-a/q)}{(i\pi)^k } 
$$ 
lies in the $q$-th cyclotomic field $\Q(\zeta_q)$.
\end{lem}

Now one can see that each generating element  
of  the arithmetic space $V_{ar}(\K)$ is actually transcendental. 
However, we call the space  $V_{ar}(\K)$
arithmetic as it still generates a one-dimensional space
over $\overline{\Q}$. This follows from \lemref{lem2}.

On the other hand, one expects all the generating elements 
of the transcendental space $V_{tr}(\K)$
to be algebraically independent and hence of dimension
$\varphi(q)/2$ over  $\overline{\Q}$. Note that the results
of the previous section supports such an expectation.

Having fixed these notations, we now consider 
the relatively accessible case,
namely when $\K \cap \Q(\zeta_q) = \Q$. In this case, we 
can prove the following lower bound for 
the dimension of $\widehat{V}_k(q,\K)$.

\smallskip
\begin{thm}\label{3}
Let $k > 1, q > 2$ be positive integers and $\K$ be a number 
field with $\K\cap \Q(\zeta_q)=\Q$. Then
\begin{equation*}
\dim_{\K}\widehat{V}_k(q,\K) \geq  \frac{\varphi(q)}{2} + 1. 
\end{equation*}
\end{thm}
\begin{proof}
By \lemref{lem1}, the  following $ \varphi(q)/2$ numbers 
$$
\frac{d^{k-1}}{dz^{k-1}} 
(\pi\cot\pi z)|_{z=a/q} , \phantom{m} 1\le a < q/2, ~~ (a,q)=1
$$
are linearly independent over $\K$ since $\K$ intersects $\Q(\zeta_q)$
trivially. Further, by \lemref{lem2}, each of these numbers 
$$
\zeta(k,a/q)+(-1)^k \zeta(k,1-a/q)
$$
is an algebraic multiple of $\pi^k$ and hence  $V_{ar}(\K)$ does
not contain $1$. Thus using the identity given by  \eqref{strong}, 
we have the lower bound mentioned
in the theorem. 
\end{proof}

Any improvement of the above lower bound 
for odd $k$ will have remarkable consequences. In 
particular, we have the following consequence 
which is not difficult to derive.

\smallskip

\begin{prop}\label{4}
Let $k>1$ be an odd integer. If 
${\rm dim}_{\K}\widehat{V}_k(4,\K)=3$ for 
all real number fields $\K$, then 
$\zeta(k)$ is transcendental. 
\end{prop}

In this context, we have the following conditional
improvement of the above lower bound for odd $k$. 

\smallskip
\begin{thm}\label{6}
Let $k>1$ be an odd integer and $q,r >2$ be two co-prime integers. Also, let 
$\K$ be a real number field with discriminant
$d_{\K}$ co-prime to $qr$. Assume that $\zeta(k) \notin \K$. Then either 
\begin{eqnarray*}
{\rm dim}_{\K}\widehat{V}_k(q,\K) &\geq& \frac{\varphi(q)}{2} + 2 \\
\text{ or } \phantom{m} 
 {\rm dim}_{\K}\widehat{V}_k(r,\K) &\geq& \frac{\varphi(r)}{2} + 2.
\end{eqnarray*} 
\end{thm}
\begin{proof}
Suppose not. Then by the above theorem, we have
\begin{equation*}
 {\rm dim}_{\K}\widehat{V}_k(q,\K)= \frac{\varphi(q)}{2}+1
\end{equation*}
and
\begin{equation*}
 {\rm dim}_{\K}\widehat{V}_k(r,\K)= \frac{\varphi(r)}{2}+1.
\end{equation*}
Now for the first case, the numbers 
\begin{equation*}
1,  ~\zeta(k,a/q)-\zeta(k,1-a/q), ~ {\rm where} ~(a,q)=1, ~ 1\leq a<q/2
\end{equation*}
generate $\widehat{V}_k(q,\K)$ over $\K$. 
Since $k$ is odd, we have 
\begin{equation}\label{6_1}
 \frac{\zeta(k,a/q)-\zeta(k,1-a/q)}{(\pi i)^k}\in \Q(\zeta_q)\subseteq \K(\zeta_q).
\end{equation}
Now consider the identity 
\begin{equation*}
\zeta(k)\underset{p ~{\rm prime}, \atop p|q}{\prod}(1-p^{-k}) ~=~
q^{-k}\sum_{\substack{a=1 \\ (a,q)=1}}^{q-1}\zeta(k,a/q) \in \widehat{V}_k(q,\K).
\end{equation*}
Thus $\zeta(k)\in \widehat{V}_k(q,\K)$ and hence 
$$
\zeta(k)= \alpha_1 + \sum_{\substack{(a,q)=1 \\ 1\leq a<q/2}}
\beta_a\left[\zeta(k,a/q)-\zeta(k,1-a/q)\right]
$$
for some $ \alpha_1, \beta_a \in \K$.
Using $\eqref{6_1}$
$$
a_1:=\frac{\zeta(k) - \alpha_1}{i\pi^k} \in \K(\zeta_q).
$$
Similarly, 
\begin{equation*}
 {\rm dim}_{\K}\widehat{V}_k(r,\K)= \frac{\varphi(r)}{2}+1
\end{equation*}
implies
\begin{equation}\label{6_3}
 a_2:= \frac{\zeta(k)-\alpha_2}{i\pi^k} \in \K(\zeta_r)
\end{equation}
with  $\alpha_2 \in \K.$ Thus,
\begin{equation*}
 a_1 i\pi^k+\alpha_1=a_2 i\pi^k+\alpha_2
\end{equation*}
which implies
\begin{equation*}
(a_1-a_2)i\pi^k =\alpha_2-\alpha_1 .
\end{equation*}
Transcendence of $\pi$ implies that $\alpha_1=\alpha_2$, $a_1=a_2$
and hence
\begin{equation*}
 \frac{\zeta(k)-\alpha_1}{i\pi^k}\in \K(\zeta_q)\cap \K(\zeta_r)=\K
\end{equation*}
because $(d_{\K}, qr) = 1$. Since $\K \subset \R$, $\zeta(k)=\alpha_1\in \K$,
a contradiction.
This completes the proof of the theorem.
\end{proof}

We end the section by proposing what we believe should be
the  extended Milnor conjecture for number fields $\K$
that intersect $\Q(\zeta_q)$ trivially, namely:

\bigskip
{\it The dimension of the $\K$-linear space
$\widehat{V}_k(q,\K)$ when $\K \cap \Q(\zeta_q) = \Q$
is equal to $\varphi(q) + 1$}.

\bigskip
When  $\K$ intersects $\Q(\zeta_q)$
non-trivially, the situation is more involved and this is
the content of the next section.

\section{Extended Milnor conjecture over 
number fields intersecting $\Q(\zeta_q)$ non-trivially}

In this section, we consider the case when
the ambient number field $\K$ intersects
$\Q(\zeta_q)$ non-trivially. The difficulty here
is that the result of Okada is no longer valid which
precludes us from concluding about the dimension
of the arithmetic space $V_{ar}(\K)$. 

Here we have the following theorem.

\begin{thm}
Let $k > 1, q > 2$ be integers. For $1 \le a < q/2, ~(a,q)=1$, 
let  $\lambda_a $ be defined as
$$
\lambda_a := \frac{\zeta(k,a/q) + (-1)^k\zeta(k,1-a/q)}{ (\pi i)^k}.
$$
If $\lambda_a \in \K$ for some $a$ as above, then
$$
2 \leq ~{\rm dim}~ \widehat{V}_k(q,\K) \leq \frac{\varphi(q)}{2}+2.
$$
\end{thm}
\begin{proof}
We first recall that (see \cite{GMR})
$$\frac{\zeta(k,a/q) +(-1)^k \zeta(k,1-a/q)}{(\pi i)^k}
= A \sum_{b=1}^{q} \left(\zeta_q^{ab} +(-1)^k \zeta_q^{-ab}\right)
B_k(b/q)$$
where $B_k(x)$ is the $k$-th Bernoulli polynomial and $A$
is a rational number.
Suppose $\lambda_a  \in \K$. Then  $\lambda_a  \in K := \K \cap \Q(\zeta_q)$.
Since $K$ is Galois (in fact abelian) over $\Q$, every element
of the Galois group $G = \text{ Gal}(\Q(\zeta_q)/\Q)$ when restricted to 
$K$ gives an automorphism
of $K$. Note that for any $(r,q)=1$, the corresponding element $\sigma_r$ of 
$G$, given by the action $\zeta_q \to \zeta_q^r$), takes
$\lambda_a$ to $\lambda_{ar}$. Hence  $\lambda_c \in K$
for all $(c,q)=1$ with  $1 \leq c < q/2$. Now the upper bound is obvious as 
$V_{ar}(\K)$ is of dimension one over $\K$ and because $1 \notin V_{ar}(\K)$.
This also gives the lower bound.
\end{proof}

As a corollary, we have
\smallskip
\begin{cor}\label{1}
For $k>1, q>2$, we have
 $$
 2 \leq ~{\rm dim}~ \widehat{V}_k(q,\Q(\zeta_q)) 
 \leq \frac{\varphi(q)}{2} + 2.
 $$
\end{cor}

To get an idea of the
difficulty, we now give an instance where the dimension
of  $V_{ar}(\K)$ does not go down even when 
 $\K$ intersects $\Q(\zeta_q)$ non-trivially. 

\begin{thm}\label{imp}
Let $k > 1, q > 2$ be a natural number and $\K = \Q(i\sqrt{d})$ 
for some square-free natural number $d \ge 1$.  If
$\K \cap \Q(\zeta_q) = \Q(i\sqrt{d})$, then
$$
{\rm dim}_{\K} ~~ V_{ar}(\K)   = \varphi(q)/2
$$ 
and thus
\begin{equation*}
\dim_{\K}\widehat{V}_k(q,\K) \geq  \frac{\varphi(q)}{2} + 1. 
\end{equation*}
\end{thm}
\begin{proof}
Write
$$ 
\lambda_a' := \zeta(k,a/q)+(-1)^k\zeta(k,1-a/q),
$$
where $(a,q)=1$ with $1 \le a < q/2$. If these numbers are linearly dependent
over $\K$,  then
$$
\sum_{a} (\alpha_a + i\sqrt{d}~\beta_a) \lambda_a' = 0,
$$
where $\alpha_a, \beta_a$ are rational numbers. Since by Okada's
theorem the numbers $\lambda_a'$'s are linearly independent over $\Q$,
we have $\alpha_a= 0 = \beta_a$ for all such $a$.
Then the theorem follows by noticing that $\pi^k \not\in \overline{\Q}$.
\end{proof}

As indicated earlier, the dimension of the space  
$\widehat{V}_k(q,\K)$ for odd
$k$ is particularly important. Here  one has the following
proposition.

\smallskip
\begin{prop}\label{2}
There exists an integer $q_0 > 2$ such 
that for all integers $q > 2$ with $(q_0, q)=1$, 
the dimension of the space $\widehat{V}_k(q,\Q(\zeta_q))$ 
is at least 3 for infinitely many odd $k$. 
\end{prop}

\begin{proof}
Suppose that for any two co-prime integers $q$ and $r$, we have 
\begin{equation*}
\dim\widehat{V}_k(q,\Q(\zeta_q))=2 ~
{\rm and} ~ \dim\widehat{V}_k(r, \Q(\zeta_r))=2.
\end{equation*}
As $k$ is an odd integer, we have
\begin{equation*}
 \zeta(k,a/q)-\zeta(k,1-a/q)\in i \pi^k\Q(\zeta_q)
\end{equation*}
for all $1\le a<q/2$ with $(a,q)=1$ and
\begin{equation*}
\zeta(k, b/r)-\zeta(k,1- b/r)\in i \pi^k\Q(\zeta_r) 
\end{equation*}
for all $1\le b < r/2$ with $(b, r)=1$.
Hence the spaces $\widehat{V}_k(q,\Q(\zeta_q))$ and 
$\widehat{V}_k(r,\Q(\zeta_r))$ are generated by $1$ and $i\pi^k$ 
over $\Q(\zeta_q)$ and $\Q(\zeta_r)$ respectively.

Again we know that $\zeta(k)$ belongs to both the
spaces $\widehat{V}_k(q,\Q(\zeta_q))$ and 
$\widehat{V}_k(r,\Q(\zeta_r))$.
Hence $\zeta(k)$ can be written as
\begin{equation}\label{*}
 \zeta(k)= \alpha_1 + \alpha_2 i\pi^k = \beta_1 + \beta_2 i\pi^k
\end{equation}
for some $\alpha_1, \alpha_2\in \Q(\zeta_q)$ 
and $\beta_1, \beta_2\in\Q(\zeta_r)$.
Thus we have 
\begin{equation*}
(\alpha_2 -  \beta_2)i\pi^k = \beta_1 - \alpha_1. 
\end{equation*}
Transcendence of $\pi$ implies that 
$\alpha_1= \beta_1$ and $\alpha_2 = \beta_2$. 
As $\Q(\zeta_q)\cap\Q(\zeta_r)=\Q$,
we see that both $\alpha_1, \alpha_2$ are rational numbers.
Then by \eqref{*}, it follows that 
$\zeta(k)$ is necessarily rational.
By the work of Rivoal \cite{TR}, we know that there 
are infinitely many
odd $k$ such that $\zeta(k)$ is irrational. 
Thus we have the proposition. 
\end{proof}

We summarise the issues involved in the number field version
of the extended Milnor conjecture. This is modelled upon 
our experience in relation to the corresponding questions 
involving the interrelation among the values of $L(1,\chi)$
as discussed in Section 2.\\

\begin{itemize}

\item{} 
It is clear that $V_{ar}(\K) \cap \overline{\Q} = \{ 0 \}$. 
However the dimension of $V_{ar}(\K)$ over $\K$ is 
most likely the only parameter which depends
 on the ambient number field $\K$. As we noticed,
 $V_{ar}(\K)$ is a one-dimensional 
vector space over  $\overline{\Q}$. The dimension
of the arithmetic space does not seem to 
have any transcendental input.
\\
\item{} 
One expects that the elements of the generating set of $V_{tr}(\K)$
are linearly independent over $\K$ and therefore have dimension
$\varphi(q)/2$. In fact, one expects this to 
hold even over  $\overline{\Q}$. This is likely to be a 
transcendental issue.
\\
\item{}
One believes that
$$
V_{tr}(\K) \cap \K = 0.
$$ 
Again, this is likely to be a transcendental issue.
\\
\item{}
Finally, one expects that  the sum
$$
V_{tr}(\K) +  V_{ar}(\K) + \K
$$ 
is direct.
But this  supposedly involves the 
question of independence between families of 
different transcendental numbers and hence may have 
both transcendental as well as arithmetic input. 
\end{itemize}

\section{Space generated by normalised Hurwitz zeta values}

In this section, we define the following new class 
of $\Q$-linear spaces.

\smallskip
\begin{defn}
For integers $k > 1 , q >  2$, let $S_k(q)$ be 
the $\Q$-linear space defined by
\begin{equation*}
S_k(q) := \Q-{ \rm span ~ of~} 
\left\{\frac{\zeta(k,a/q)}{\pi^k}:~ 1\leq a <q, ~(a,q)=1\right\}
\end{equation*}
and $\widehat{S}_k(q)$ be the $\Q$-linear space defined by
\begin{equation*}
\widehat{S}_k(q) := \Q-{ \rm span ~ of~} 
\left\{1, ~\frac{\zeta(k,a/q)}{\pi^k}:~ 1\leq a <q, ~(a,q)=1\right\}.
\end{equation*}
\end{defn}
These spaces appear similar to the spaces related to Milnor
and extended Milnor conjecture respectively. But there
is an important distinction, namely the parity of $k$ enters the picture.
Recall, the conjectural dimension of the extended Milnor spaces is 
independent of parity of $k$. But this is no longer the case for these
new spaces.
 
However as before, in relation to these spaces also, we can deduce the 
following lower bound.

\begin{thm}\label{**}
Let $k>1$ and $q>2$ be two integers. Then 
\begin{equation*}
\dim_{\Q}S_k(q) \geq  \frac{\varphi(q)}{2}. 
\end{equation*}
\end{thm}

\begin{proof} 
First note that the space $S_k(q)$ is also spanned by the following sets of real numbers:
\begin{equation*}
\left\{\frac{\zeta(k,a/q)+\zeta(k,1-a/q)}{\pi^k}|~(a,q)=1, ~ 1\leq a<q/2\right\},
\end{equation*}
\begin{equation*}
 \left\{\frac{\zeta(k,a/q)-\zeta(k,1-a/q)}{\pi^k}|~(a,q)=1, ~ 1\leq a<q/2\right\}. 
\end{equation*}
Then, again by  the following ubiquitous identity 
\begin{equation*}
\zeta(k,a/q)+(-1)^k \zeta(k,1-a/q)=\frac{(-1)^{k-1}}{(k-1)!}\frac{d^{k-1}}{dz^{k-1}} 
(\pi\cot\pi z)|_{z=a/q} 
\end{equation*}
and by the result of  Okada, the numbers on the right hand side for 
$1\leq a< q/2$ with $(a,q)=1$  are $\Q$-linearly independent.
 Hence the following numbers
$$
 \frac{\zeta(k,a/q)+(-1)^k \zeta(k,1-a/q)}{\pi^k}, 
\phantom{mm} 1\leq a< q/2,~~ (a,q)=1 $$  are linearly independent over $\Q$.
\end{proof}

Interestingly, the parity of $k$ enters the picture non-trivially as
seen by the following proposition.

\begin{thm}
 Let $k>1$ be an even integer and $q>2$ be any integer. 
Then $S_k(q)=\widehat{S}_k(q)$.
\end{thm}

\begin{proof} 
 Note that for even $k$,
\begin{equation*}
\sum_{\substack{a=1\\ (a,q)=1}}^{q-1}\frac{\zeta(k,a/q)}{\pi^k}=q^k
\underset{p ~{\rm prime}, 
\atop  p|q}{\prod}(1-p^{-k})\frac{\zeta(k)}{\pi^k}\in \Q.
\end{equation*}
Hence for $k$ even, $$\Q\subset S_k(q)$$ and thus 
$S_k(q)=\widehat{S}_k(q)$.
\end{proof}

Thus for an even $k$, $\Q$  lies 
in the associated normalised arithmetic space.
However, when $k$ is an odd integer, we expect the picture 
to be different. For instance, unlike the earlier case, 
$\Q$ does not seem to belong to the  normalised arithmetic space,
at least when $4 \nmid q$. 

\smallskip

\smallskip
\begin{thm}
 Let $k>1$ be an odd integer and $4 \nmid q$. 
Then $\Q$ does not belong to the normalised arithmetic space,
that is, the $\Q$-vector space generated
by the numbers
 $$
 \frac{\zeta(k,a/q) - \zeta(k,1-a/q)}{\pi^k}, \phantom{m} 1\leq a< q/2,
~(a,q)=1. 
$$
intersects $\Q$ trivially.
\end{thm}
\begin{proof}
Suppose that $1$ belongs to the given space.
As noted before,   each of these numbers
 $$
 \frac{\zeta(k,a/q) - \zeta(k,1-a/q)}{\pi^k},
 \phantom{m}1\leq a< q/2, (a,q)=1 
 $$
when multiplied by $i$ lie inside the 
$q$-th cyclotomic field.  Therefore, if
$1$ is expressible as a rational linear combination
of these numbers, then $i$ necessarily lies in
the $q$th cyclotomic field. This not possible
as $4 \nmid q$. This completes the proof.
\end{proof}

Further, when $k$ is odd, we can also derive the following result
by employing the earlier techniques as in \propref{2}.

\smallskip
\begin{thm}
 Let $k>1$ be an odd integer. Then there exists a $q_0 > 2$ such that
\begin{equation*}
{\rm dim}_\Q S_k(q)\geq \frac{\varphi(q)}{2}+1
\end{equation*}
for any $q > 2$ co-prime to $q_0$.
\end{thm}

To conclude, while 
$$
S_k(q)=\widehat{S}_k(q)
$$ 
when $k$ is even, there is reason to believe that
$$
S_k(q) \subsetneq \widehat{S}_k(q)
$$
when $k$ is odd, at least when $4 \nmid q$.

\bigskip
\noindent
{\bf Acknowledgments.}  It is our pleasure to thank Ram Murty for
several suggestions in relation to an earlier version of the paper.
The last two authors would like to thank ICTP, Trieste for the hospitality 
extended to them during their visit as associates  where this work was initiated.

\end{document}